\documentclass[preprintnumbers,amsmath,12pt,amssymb,floatfix,superscriptaddress,nofootinbib]{article}
\usepackage{amsmath,amssymb,latexsym,textcomp,mathrsfs}
\usepackage[all]{xy}
\usepackage{graphicx}
\usepackage{bm,amsmath, amsthm, amssymb, amsfonts}
\usepackage{times}
\usepackage[english]{babel}
\usepackage{setspace}
\setbox0=\hbox{$+$}
\newdimen\plusheight
\plusheight=\ht0
\def\+{\;\lower\plusheight\hbox{$+$}\;}

\setbox0=\hbox{$-$}
\newdimen\minusheight
\minusheight=\ht0
\def\-{\;\lower\minusheight\hbox{$-$}\;}

\setbox0=\hbox{$\cdots$}
\newdimen\cdotsheight
\cdotsheight=\plusheight
\def\cds{\lower\cdotsheight\hbox{$\cdots$}}

\numberwithin{equation}{section}
\theoremstyle{plain}
\newtheorem{theorem}{Theorem}[section]
\newtheorem{lemma}{Lemma}[section]

\newtheorem{example}{Example}[section]
\newtheorem{definition}{Definition}[section]

\newtheorem{proposition}{Proposition}[section]
\newtheorem{remark}{Remark}[section]

\newtheorem{note}{Note}[section]

\def\mytitle#1{\setcounter{equation}{0}
\setcounter{footnote}{0}
\begin{flushleft}\Large\textbf{#1}\end{flushleft}
\vspace{0.20cm}}
\def\myname#1{\leftline{{\large #1}}\vspace{-0.13cm}}
\def\myplace#1#2{\small\begin{flushleft}\textit{#1}\\
\texttt{#2}\end{flushleft}}

\def\myclassification#1{\small\noindent
Keywords : $I^K_\nu$-Convergence, AP$(I,K)$ Condition, $(I\vee K)^K_\nu$-Convergence, $I^K_\nu$-Cauchy, $I^K_\nu$-Limit Points.
       #1\vspace{0.5cm}#1\vspace{0.5cm}\\
       AMS subject classification$(2020)$: Primary: 54E70; Secondary: 40A05; 40A35}
\usepackage{graphicx}
\usepackage{amsmath}

\begin{document}
\mytitle{$I^K_{\nu}$-Convergence of functions in probabilistic normed spaces}

\myname{$Amar Kumar Banerjee^{\dag}$\footnote{akbanerjee@math.buruniv.ac.in, akbanerjee1971@gmail.com} and  $Mahendranath~ Paul^{\dag}$\footnote{mahendrabktpp@gmail.com}}
\myplace{$\dag$Department of Mathematics, The University of Burdwan, Purba Burdwan -713104, India.} {}

\begin{abstract}
In this paper we study $I^K$-convergence of functions with respect to probabilistic norm $\nu$ which is a generalization of $I^*_{\nu}$-convergence in probabilistic norm spaces. We also study on $I^K$-Cauchy functions and $I^K$-limit points with respect to  probabilistic norm $\nu$ in the same space. 
\end{abstract}
\myclassification{}
\section{Introduction and background}
Menger \cite{MK} initially introduced the concept of probabilistic metric spaces in 1942 as a captivating and significant extension of a metric space. He referred to it as a statistical metric. Subsequently, Schweizer and Sklar \cite{SB, SB1} delved deeper into this subject and conducted further investigations on it. In the theory of probabilistic metric spaces, the usual concept of distance takes a probabilistic nature. This led to the development of the idea of probabilistic metric spaces. In \cite{SB, SB1} Schweizer and Sklar proposed the probabilistic metric replacing the usual distance measure $d(\alpha,\beta)$ by a real-valued function $F_{\alpha\beta}$. The function $F_{\alpha\beta}(\gamma)$ represents the probability that the distance between elements $\alpha$ and $\beta$ is less than a given value $\gamma$. In 1962, Sherstnev \cite{SA} introduced the concept of probabilistic normed spaces, which serves as a generalization of the established results in linear normed spaces. He highlighted the significance of probabilistic normed spaces and concluded that they represent an important family within the domain of probabilistic metric spaces. Since then, numerous authors have dedicated their research efforts to study this concept. For more detailed information, please refer to \cite{GLB, HB}.
On the other side in the early sixties of twentieth century there was considerable attention given to the generalization of convergence of sequences. The concept of convergence for real sequences so called usual convergence, was extended to statistical convergence by H. Fast \cite{FH} as well as independently by H. Steinhaus \cite{SH}. Subsequently, this notion was further developed by various authors such as \cite{BM,BA7,FJA1,KP0, MuM1}. In the year 2000, the notion of $\lambda$-statistical convergence was introduced by Mursaleen in \cite{MuM2} as an extension of statistical convergence. Another notable generalization  of statistical convergence is the concept of ideal convergence (namely $I$ and $I^*$-convergence) which is based on the structure of ideals formed by subsets of the natural numbers introduced by P.Kostyrko \cite{KP1} in the beginning of twenty first century. Subsequently the concept of $I$-convergence was further extended by numerous authors  to the metric spaces, normed linear spaces, cone metric spaces, $S$-metric spaces etc. In $2005$, B. K. Lahiri and P. Das \cite{LB2} extended the notion of $I$ and $I^*$-convergence in a topological space and they revealed that the fundamental properties of these convergence are also preserved within the context of topological spaces.  The notion of $I$-Cauchy sequences was first introduced by K. Dems in 2004 in \cite{DmK}. Subsequently, in 2007, A. Nabiev et al. \cite{NA} conducted further investigations on $I^*$-Cauchy sequences. Since then, this concept has garnered significant attention from various authors. Notable references on this topic include \cite{BA8, DP1, DP2, DP3, DP5, PS}.\\
In 2007, S. Karakus \cite{KS} extended the concept of statistical convergence to probabilistic normed spaces (in short PNS or PN Space) and observed that certain fundamental properties of statistical convergence are preserved in probabilistic normed spaces but some  properties do not hold in such spaces due to the presence of the triangle function. M. Musaleen et al. \cite{MuM} and Rahamat et al. \cite{RSRM} defined and studied the concept of ideal convergence in probabilistic normed spaces. V. Kumar et al. \cite{KV} introduced the notion of $I$-Cauchy and $I^*$-Cauchy sequences in probabilistic normed spaces.\\
In the year 2010, M. Macaj and M. Sleziak \cite{MM1} proposed the notion of  $I^K$-convergence in a topological space where $I$ and $K$ both are ideals of an arbitrary non empty set $S$ and shown that this types of convergence serves as a unified generalization that encompasses various types of $I$-convergence and $I^*$-convergence under certain restrictions. They also introduced a modified version of the AP condition \cite{KP1}, referred to as the AP$(I,K)$ condition, in their work. In the year 2014, P. Das et al. \cite{DP4} studied on $I^K$-Cauchy functions. Notable references on this topic include \cite{BA0, BA0.1, BA0.2}. \\
In recent years, several authors have explored the generalizations of convergence through the use of ideal in the context of probabilistic normed spaces. Given the advancements in this area, it appears reasonable to consider extending these investigations to include double ideals within the same space. In doing so, it becomes necessary to investigate the extent to which various fundamental properties, including limit points, cluster points, and the Cauchy condition, are affected by this extension. In this paper we study on $I^K$-convergence of functions with respect to probabilistic norm $\nu$ in a probabilistic normed space which also generalizes the idea of $I_\nu^*$-convergence in \cite{MuM}. Due to the inherent importance of convergence in probabilistic normed spaces (PNS) for probabilistic analysis, we recognize that the concept of convergence via double ideals in a PNS would provide a more comprehensive framework for analyzing PNS. This extension would offer a broader scope and enhance our understanding of convergence in PNS.

\section{Basic Definitions and Notation}
Let us revisit some fundamental definitions for our present discussion. Let $\mathbb{R}_0^+=\{x\in \mathbb{R}: x\geq 0\}$ then  distribution function is defined as follows:
\begin{definition} \cite{SB1}
	A non-decreasing function $F:\mathbb{R}\rightarrow \mathbb{R}_0^+$ with $F(-\infty)=0$ and $F(\infty)=1$ is called distribution function if it is left continuous.
\end{definition}
The notation $\bigtriangleup$ denotes the set of all left continuous distribution functions.
\begin{definition} \cite{SB}
	A triangular norm ( in short t-norm) is a binary operation $*:[0,1]\times[0,1]\rightarrow [0,1]$ that fulfills the following conditions for all $a,b,c,d\in [0,1]$\\
	(1) $*$ has neutral element i.e. $a*1=a$\\
	(2) $*$ is  commutative i.e. $a*b=b*a$\\
	(3) $*$ is non decreasing i.e. $c*d\geq a*b$ if $c\geq a$ and $d\geq b$\\
	(4) $*$ is associative i.e. $(a*b)*c=a*(b*c)$
	
\end{definition}
\begin{example}
	The $*$ operation $s*t=$ max$\{s+t-1\}$, $s*t=st$ and $s*t=$ min $\{s,t\}$ on $[0,1]$ are $t$-norm.
\end{example}
\begin{definition} \cite{SB}
	A triplet $(P,\nu, *)$ is called a probabilistic normed space (for short PNS or PN Space) if $P$ is a real linear space and $\nu:P\rightarrow \bigtriangleup$ is a mapping (for $x\in P$, the distribution function(d.f) $\nu(x)$ is denoted by $\nu_x$ and $\nu_x(s)$ is the value of $\nu_x$ at $s \in \mathbb{R}$) and * is a $t$-norm and $\nu$ fulfills the following conditions:\\
	(i) $\nu_x(0)=0$\\
	(ii) $\nu_x(t)=1$ for all $t>0$ iff $x=0$\\
	(iii) $\nu_{\alpha x}(t)=\nu_x(\frac{1}{|\alpha|})$ for all $\alpha \in \mathbb{R}\setminus \{0\}$ and for all $t>0$.\\
	(iv) $\nu_{x+y}(s+t)\geq \nu_x(s)*\nu_y(t)$ for all $x,y\in P$ and $s,t\in \mathbb{R}$.
\end{definition}
\begin{definition}\label{6chd1}
	A function $f:S\rightarrow P$ is said to be convergent to $\xi$ with respect to probabilistic norm $\nu$ if for each $\epsilon(>0)$ and $t(>0)$ the set
	$A(\epsilon, t)=\{s\in S:\nu_{f(s)-\xi}(t)\leq 1-\epsilon\}$ is a finite set.
\end{definition}

\begin{definition}
	Let $S$ be a non-void set then a family of sets $I\subset 2^S$ is said to be an ideal if 
	\item (i) $A,B\in I \Rightarrow A\cup B\in I$
	\item(ii) $A\in I, B\subset A \Rightarrow B\in I$
\end{definition}
$I$ is called nontrivial ideal if $S\notin I$ and $ I\neq \{\phi\} $. In view of condition (ii), it is clear that $\phi\in I $. If $I \subsetneq 2^S$ we say that $I$ is proper ideal on $S$. Several examples of non-trivial ideals are seen in \cite{KP1}.
A nontrivial ideal $I$ is called admissible if it contains all the singleton of $\mathbb N$. A nontrivial ideal $I$ is called non-admissible if it is not admissible.
An example of an admissible ideal on a set $S$ is the ideal of all finite subsets of $S$ which we shall denote by Fin$(S)$. If $S=\mathbb N$ then we write Fin instead of Fin$(\mathbb N)$ for short. If $f$ is Fin-convergent to $p$ with respect to probabilistic norm $\nu$, then we denote Fin$_\nu(S)$-$\lim f=\xi$. 
\begin{example}
	Let $I$ be the class of all $A\subset \mathbb N$ with $d(A)=0$. Then $I$ is an admissible ideal of $\mathbb{N}$, since singleton sets has density zero. For any proper subset $M\subset \mathbb N$, $I=2^M$ is an non-admissible ideal of $\mathbb{N}$.
\end{example}
\begin{note}
	The dual notion to the ideal is the notion of the filter i.e. a filter on $S$ is non-void system of subsets of $S$, which is closed under finite intersection and supersets.
	If $I$ is a non-trivial ideal on $S$ then $F=F(I)=\{E\subset S:S\setminus E \in I \}$ is clearly a filter on $S$ and conversely. The filter $F(I)$ is called associated filter with respect to ideal $I$.
\end{note}
\begin{definition} \cite{DP4}
	Let $I,K$ be ideals on the non-empty set $S$. We say that $I$ has additive property with respect to $K$ or that the condition AP$(I,K)$ holds if for every sequence of pairwise disjoint sets $A_n\in I$, there exists a sequence $B_n\in I$ such that $A_n\bigtriangleup B_n\in K$ for each $n$ and $\displaystyle \cup_{n\in \mathbb N}B_n\in I$.
\end{definition}
\section{$I_\nu^K$-Convergence of functions}
Throughout  the paper $P$ stands for a probabilistic norm space(briefly PNS or PN Space) and we always assume that  $I$, $K$ are non-trivial ideals of a non empty set $S$ unless otherwise stated. First we give the definition of $I$-convergence of functions instead of sequences in PNS.
\begin{definition}
	Let $I$ be an ideal on a non-empty set $S$ and $(P,\nu,*)$ be a probabilistic normed space. A function $f:S\rightarrow P$ is said to be $I$-convergent with respect to norm $\nu$ to $p\in P$ if for each $\epsilon (>0)$ and $t(>0)$
	$$\{s\in S:\nu_{f(s)-p}(t)\leq 1-\epsilon \}\in I$$ .
\end{definition}
We denote $I_\nu$-$ \lim f=p$, $p$ is called the $I$-limit of the function $f$ with respect to probabilistic norm $\nu$. If $I$=Fin($\mathbb{N}$), then $I_\nu$-convergence coincides with the Definition \ref{6chd1} of usual convergence with respect to probabilistic norm $\nu$.
\begin{theorem}
	Let $(P,\nu,*)$ be a probabilistic normed space and $f:S\rightarrow P$ be a function. Then the following statements are equivalent:
	\item (i) $I_\nu$-$\lim f=p$
	\item (ii) $\{s\in S:\nu_{f(s)-p}(t)> 1-\epsilon \}\in F(I)$.
	\item (iii) $I-\lim \nu_{f(s)-p}(t)=1$
\end{theorem}
\begin{proof}
	Suppose that $f$ is $I_\nu$-convergent to $p$ then for each $\epsilon > 0$ and $t > 0$ the set $\{s\in S:\nu_{f(s)-p}(t)\leq 1-\epsilon \}\in I$ i.e. $\{s\in S:\nu_{f(s)-p}(t)> 1-\epsilon \}\in F(I)$. So (i) implies (ii). Again consider the set $\{s\in S:|\nu_{f(s)-p}(t)-1|\geq \epsilon \}= \{s\in S:\nu_{f(s)-p}(t)\geq 1+\epsilon \} \cup \{s\in S:\nu_{f(s)-p}(t)\leq 1-\epsilon \} $. Since for any $\epsilon > 0$ the set  $\{s\in S:\nu_{f(s)-p}(t)\geq 1+\epsilon \}=\phi \in I $ and $\{s\in S:\nu_{f(s)-p}(t)\leq 1-\epsilon \} \in I$ (as from (ii) we have  $\{s\in S:\nu_{f(s)-p}(t)> 1-\epsilon \}\in F(I)$ ). Therefore the set $\{s\in S:|\nu_{f(s)-p}(t)-1|\geq \epsilon \}\in I$. Hence (ii) implies (iii). Converse part follows in similar technique. Hence the result holds.  
\end{proof}.
If $S=\mathbb N$ we obtain the usual definition of $I_\nu$-convergence of sequences in PN Spaces. In this case the notation $I_\nu$-$\lim p_n=p$ is used for a real sequence $\{p_n\}$. \\
We consider some basic results regarding $I_\nu$-convergence in a PNS.
\begin{lemma}\label{6chl1}
	Let $I, K$ be ideals on an arbitrary set  $S$ and let $P$ be a PNS.\\
	(i) $I_\nu$-limit of a function is unique.\\
	(ii) If $I\subseteq K$ then for each function $f:S\rightarrow P$, we get  $I_\nu$-$\lim f=p$ implies $K_\nu$-$\lim f=p$.\\
\end{lemma}
\begin{proof}
	(i) The proof is parallel with the proof of Theorem 2.2 in \cite{MuM}. (ii) Proof is trivial.
\end{proof}
Since we are working with functions, we modify the definition of $I_\nu^*$-convergence in PN Spaces.
\begin{definition}
	Let $I$ be an ideal on an arbitrary set $S$ and let $f:S\rightarrow P$ be a function in PNS $P$. The function $f$ is called $I^*$-convergent to $p \in P$ with respect to probabilistic norm $\nu$ (in short $I_\nu ^{*}$-convergent to $p$)  if there exists a set $M\in F(I)$ such that the function $g:S\rightarrow P$ defined by
	\[g(s)=\left\{\begin {array}{ll}
	f(s) & \mbox{if $s\in M$} \\
	p & \mbox{if $s\notin M$}
\end{array}
\right. \] \\
is Fin$(S)$-convergent to $p$ with respect to norm $\nu$. 
\end{definition} 
If $f$ is $I^*$-convergent to $p$ with respect to norm $\nu$ (i.e., $I^*_\nu$-convergent to $p$), then we write $I^*_\nu$-$\lim f= p$. The usual notion of $I^*_\nu$-convergence of a sequence is a special case when $S=\mathbb N$, and we write $I^*_\nu$-$\lim p_n=p$ for a real sequence $\{p_n\}$.\\
In the definition of $I^K$-convergence with respect to probabilistic norm $\nu$ we simply replace the Fin by an ideal $K$ on the set $S$.
\begin{definition}
Let $K$ and $I$ be ideals on an arbitrary set $S$ and let $P$ be a PNS and $p$ be an element of $P$. The function $f:S\rightarrow P$  is called $I^K$-convergent to $p\in P$ with respect to probabilistic norm $\nu$ (in short $I_\nu ^{K}$-convergent to $p$) if there exists a set $M\in F(I)$ such that the function $g:S\rightarrow P$ given by
\[g(s)=\left\{\begin {array}{ll}
f(s) & \mbox{if $s\in M$} \\
p & \mbox{if $s\notin M$}
\end{array}
\right. \] \\
is $K_\nu$-convergent to $p$ (i.e. $K$-convergent to $p$ with respect to probabilistic norm $\nu$.)

\end{definition}
If $f$ is $I^K$-convergent to $p$ with respect to probabilistic norm $\nu$ (i.e., $I^K_\nu$-convergent to $p$), then we write $I^K_\nu$-$\lim f= p$.
As usual, notion of $I_\nu^K$-convergence of sequences is a special case when $S=\mathbb N$.
\begin{remark}
We can reformulate the definition of $I_\nu^K$-convergence in the following way: if there exists $M\in F(I)$ such that the function $f|_M$ is $K|_M$-convergent to $p$ with respect to probabilistic norm $\nu$ where $K|_M=\{A\cap M:A\in K\}$.
\end{remark}

\begin{remark}
$I_\nu^K$-convergence as a generalization of all types of $I_\nu^*$-convergence of sequences and functions from $S$ to $P$. Here we will work with functions. One of the reasons is that utilizing functions can be advantageous as it often leads to a simplification of notation.
\end{remark}
\begin{lemma}\label{6ch0}
If $I$ and $K$ are ideals on an arbitrary set $S$ and $f:S\rightarrow P$ is a function such that $K_\nu$-$\lim f= p$, then $I^K_\nu$-$\lim f= p$.
\end{lemma}
\begin{proof}
If we consider the set $M=S$ then it is obvious that $M\in F(I)$. Therefore for this case ideals $K|_M$ and $K$ both are equal. Hence $f$ is also $K|_M$-convergent to $p$ with respect to probabilistic norm $\nu$. This shows that $f$ is $I^K_{\nu}$-convergent to $p$.
\end{proof}
\begin{proposition}\label{6chpro1}
Let $I,J,K$ and $L$ be ideals on a set $S$ such that $I\subseteq J$ and $K\subseteq L$ and let $P$ be a PNS. Then for any function $f:S\rightarrow P$, we have\\
(i) $I^K_\nu$-$\lim f= p ~~~ \Rightarrow ~~~ J^K_\nu$-$\lim f= p$ 
(ii) $I^K_\nu$-$\lim f= p ~~~ \Rightarrow ~~~ I^L_\nu$-$\lim f= p$
\end{proposition}
\begin{proof}
(i) As $I^K_\nu$-$\lim f= p$ so there exist a set $M\in F(I)$ such that the function $g:S\rightarrow P$ given by
\[g(s)=\left\{\begin {array}{ll}
f(s) & \mbox{if $s\in M$} \\
p & \mbox{if $s\notin M$}
\end{array}
\right. \] \\
is $K_\nu$-convergent to $p$. Here $M\in F(I)\subseteq F(J)$ as $I\subseteq J$. So obviously $J^K_\nu$-$\lim f= p$.\\
(ii) The proof directly follows from the fact $K \subseteq L$ and Lemma \ref{6chl1}(ii).

\end{proof}

\begin{theorem}\label{6chth2}
Let $I$, $K$ be ideals on an arbitrary set $S$,let $P$ be a PNS and let $f$ be a function from $S$ to $P$ then
(i) $I^K_\nu$-$\lim f= p ~~~ \Rightarrow ~~~ I_\nu$-$\lim f= p$ if $K\subseteq I$. (ii)  $ I_\nu$-$\lim f= p ~~~ \Rightarrow ~~~ I^K_\nu$-$\lim f= p$ if $I\subseteq K$
\end{theorem}
\begin{proof}
(i) Now  $I^K_\nu$-$\lim f= p$, then by the definition of $I_\nu^K$-convergence there exist a set $M\in F(I)$ such that the function $g:S\rightarrow P$ given by
\[g(s)=\left\{\begin {array}{ll}
f(s) & \mbox{if $s\in M$} \\
p & \mbox{if $s\notin M$}
\end{array}
\right. \] \\
is $K_\nu$-convergent to $p$. i.e. for every $\epsilon(>0)$ and $t(>0)$
$$M_{\nu,g}=\{s\in S: \nu_{g(s)-p}(t)\leq 1-\epsilon\}\in K\subseteq I$$
Now for any $t(>0)$ and $\epsilon(>0)$, we consider the set
$$M_{\nu,f}=\{s\in S: \nu_{f(s)-p}(t)\leq 1-\epsilon\}$$
Again, $M_{\nu,g}= M_{\nu,f}\cap M\in K\subseteq I$. Now $M_{\nu,f}\subseteq (S\setminus M)\cup M_{\nu,g}$. Since $(S\setminus M)\in I$ and $M_{\nu,g}\in I$. So $M_{\nu,f}=\{s\in S: \nu_{f(s)-p}(t)\leq 1-\epsilon\}\in I$. Thus $I_\nu$-$\lim f= p$.\\
(ii) The proof follows from the Lemma \ref{6chl1} (ii) and Lemma \ref{6ch0} 
\end{proof}
\begin{note}
If  $K\not\subset I$ and $I\not\subset K$ then the implications (i) in theorem \ref{6chth2} may not be true. To establish this we cite an example which is $I^K_\nu$-convergence but not $I_\nu$-convergence.
\end{note} 
\begin{example}
Consider a set $B\in K\setminus I$ and two distinct point $x,y \in P$. Let us define the function $f:S\rightarrow P$ by
\[f(s)=\left\{\begin {array}{ll}
x & \mbox{if $s\in S\setminus B$} \\
y & \mbox{if $s\in B$}
\end{array}
\right. \] \\
Now we consider the set $M_{\nu,f}=\{s\in S: \nu_{f(s)-x}(t)> 1-\epsilon\}$. So $S\setminus B \subseteq M_{\nu,f}$ i.e.  $M_{\nu,f} \in F(K)$ as $S\setminus B \in F(K)$. Therefore $f$ is $K_\nu$-convergent to $x$. Hence by the Lemma \ref{6ch0} $f$ is $I_\nu ^K$-convergent to $x$ i.e. $I^K_\nu$-$\lim f=x$. Now if $\nu_{y-x}(t)> 1-\epsilon $ then for arbitrary $\epsilon > 0$ we get $\nu_{(y-x)}(t)=1$, for all $t> 0$. So we get $y=x$ which contradicts that $x$ and $y$ are distinct elements. Therefore $\nu_{y-x}(t) \leq 1-\epsilon $. So $M_{\nu,f}=(S\setminus B)$. Therefore $M_{\nu,f}^c=\{s\in S: \nu_{f(s)-x}(t)\leq 1-\epsilon\}=B \notin I$. Hence $f$ is not $I_\nu$-convergent to $x$. So $I_\nu \lim f\ne x$.

\end{example}
\begin{note}
Consider the two sets $A=\{3n:n\in \mathbb{N}\}$ and $B=\{4n:n\in \mathbb{N}\}$ then $2^{A}$ and $2^{B}$ are two ideals such that $2^{A}\not\subset2^{B}$ and  $2^{B}\not\subset2^{A}$ but  $2^{A}\cap 2^{B}\neq \phi$.
\end{note}

\subsection{$I_\nu ^I$, $K_\nu ^I$ and $(I\vee K)_\nu^K$-Convergence}
In this part we discuss the case when $I=K$ and for any two ideals  $I, K$ on a non void set $S$ we have the ideal $$I\vee K=\{A\cup B:A\in I,B\in K\}$$
which is the smallest ideal containing both $I$ and $K$ on $S$ i.e. $I,K\subseteq I\vee K$. It is clear that if $I\vee K$ is non-trivial then $I$ and $K$ both are non-trivial. But converse part may not be true. To support this following examples are given.
\begin{example}
Consider the two sets $N_1=\{4n:n\in \mathbb{N}\}$ and $N_2=\{4n-1:n\in \mathbb{N}\}$ now it is clear that $2^{N_1}$, $2^{N_2}$ and $2^{N_1}\vee 2^{N_2}$ all are non-trivial ideal on $\mathbb{N}$
\end{example}

\begin{example}
Now let  $N_1$ be set of all odd integers and $N_2$ be set of all even integers. Then it is clear that $I=2^{N_1}$ and $K=2^{N_2}$ are both non-trivial ideals on $\mathbb{N}$ but $I\vee K=2^\mathbb{N}$, the trivial ideal on $\mathbb{N}$
\end{example}
If $I\vee K$ is a non-trivial on a non empty set $S$ then the dual filter is the filter given by $$F(I\vee K)=\{G\cap H:G\in F(I),H\in F(K)\}$$

\begin{theorem}
Let $f:S\rightarrow P$ be a map, $I,K$ be ideal on the arbitrary set $S$ and $P$ be a PN space. Then\\
(i) $I_\nu$-$\lim f= p$ if and only if $I^I_\nu$-$\lim f= p$ and\\
(ii) $I^K_\nu$-$\lim f= p$ if and only if $(I\vee K)^K_\nu$-$\lim f= p$.
\end{theorem}
\begin{proof}
(i) Suppose that $f$ is $I_\nu$-convergent to $p$. Then for $M=S$ we have $I|_M=I$ which implies that $f$ is also $I^I_\nu$-convergent to $p$.\\
Conversely, let $f$ is $I^I_\nu$-convergent to $p$ so there is a set $M\in F(I)$ such that $f|_M$ is $I_\nu|_M$-convergent to $p$. Now for any $t(>0)$ and $\epsilon(>0)$, we consider the set $M_{\nu,f}=\{s\in S: \nu_{f(s)-p}(t)\leq 1-\epsilon\}$. Then there exists $G\in F(I)$ such that
$$M_{\nu,f}\cap M=G\cap M$$
Clearly $G\cap M\in F(I)$ and $G\cap M\subseteq M_{\nu,f}$. So $M_{\nu,f}\in F(I)$\\
(ii) Suppose that $f$ is $I_\nu^K$-convergent to $p$. Then there is a set $M\in F(I)$ such that $f|_M$ is $K_\nu|_M$-convergent. Clearly if $M\in F(I)$ then $M\in F(I\vee K)$. Therefore $f$ is also $(I\vee K)_\nu^K$-convergent to $p$.\\
Conversely, let $f$ is $(I\vee K)_\nu^K$-convergent to $p$ i.e. there is a set $M\in F(I\vee K)$ such that $f|_M$ is $K_\nu|_M$-convergent. Now for any $t(>0)$ and $\epsilon(>0)$, we consider the set $M_{\nu,f}=\{s\in S: \nu_{f(s)-p}(t)\leq 1-\epsilon\}$. Then there exists $G\in F(K)$ such that
$$M_{\nu,f}\cap M=G\cap M$$
Since $M\in F(I\vee K)$, then $M=M_1\cap M_2$ for some $M_1\in F(I)$ and $M_2\in F(K)$. Now we have $$M_{\nu,f}\cap M_1\supseteq M_{\nu,f}\cap M=(G\cap M_2)\cap M_1$$
Since $G\cap M_2\in F(K)$, this shows that $M_{\nu,f}\cap M_1\in F(K|_{M_1})$. So $f$ is $I_\nu^K$-convergent to $p$.
\end{proof}
\begin{remark}
Consider the two  admissible ideals $I , K$ on a non void set $S$ and $(P,\nu,*)$ be a PN Space and $f:S\rightarrow P$ be a function. Then $K_\nu^I$-$\lim f =l $ implies  $I_\nu^K$-$\lim f= l$ if $I\subseteq K$.
\end{remark}
\begin{theorem}
Let $I , K, I_1, I_2, K_1, K_2$ be  admissible ideals on a non empty set $S$ and  $f:S\rightarrow P$ be a function on a PN Space  $(P,\nu,*)$ such that \\
(i) $I_\nu^{K_1} \lim f = l$ then $I_\nu^{K_1 \vee K_2} \lim f = l$  and if moreover $I_\nu^{K_2}$-limit exist then $I_\nu^{K_2} \lim f = l$.
(ii)  $I_{1_\nu}^{K} \lim f = l$ then $(I_1 \vee I_2)_\nu^{K} \lim f = l$  and if moreover $I_{2_\nu}^{K}$-limit exist then $I_{2_\nu}^{K} \lim f = l$.
\end{theorem}
\begin{proof}
(i)	 As $K_1 \subseteq K_1 \vee K_2 $ so using the Proposition \ref{6chpro1}(ii) we have $I_\nu^{K_1} \lim f = l$ implies $I_\nu^{K_1 \vee K_2} \lim f = l$. Hence the first part of the proof is done.\\
For 2nd part of this proof, let us assume the contrary. Let $I_\nu^{K_2}$-limit exist and $I_\nu^{K_2} \lim f = m$, where $m \ne l$. As $K_2 \subseteq K_1 \vee K_2 $ so using the Proposition \ref{6chpro1}(ii) we have $I_\nu^{K_1 \vee K_2} \lim f = m$ which contradicts the fact of uniqueness of $I_\nu^{K}$-limit. So our assumption is wrong. Therefore  $I_\nu^{K_2} \lim f = l$.\\
(ii)	 As $I_1 \subseteq I_1 \vee I_2 $ so using the Proposition \ref{6chpro1}(i) we have $I_{1_\nu}^{K} \lim f = l$ implies $(I_1 \vee I_2)_\nu^{K} \lim f = l$. Hence the first part of the proof is done.\\
For 2nd part of this proof, let us assume the contrary. Let $I_{2_\nu}^{K}$-limit exist and $I_{2_\nu}^{K} \lim f = m$, where $m \ne l$. As $I_2 \subseteq I_1 \vee I_2 $ so using the Proposition \ref{6chpro1}(i) we have $(I_1 \vee I_2)_\nu^{K} \lim f = m$ which contradicts the fact of uniqueness of $I_\nu^{K}$-limit. So our assumption is wrong. Therefore  $I_{2_\nu}^{K} \lim f = l$.	
\end{proof}
\begin{theorem}
Let $I , K$ be two admissible ideals on a non empty set $S$ and  $f:S\rightarrow P$ be a function on PN Space  $(P,\nu,*)$ such that $I_\nu^{K} \lim f = l$  then \\
(i)  $(I\vee K)_\nu^{I \vee K} \lim f = l$  and if moreover $K_\nu^{I}$-limit exist then $K_\nu^{I} \lim f = l$.
(ii)  $(I\vee K)_\nu \lim f = l$  and if moreover $K_\nu^{I}$-limit exist then $K_\nu^{I} \lim f = l$.
\end{theorem}
\begin{proof}
As $K \subseteq I \vee K $ so using the Proposition \ref{6chpro1}(ii) we have  $I_\nu^{K} \lim f = l$  implies $I_\nu^{I \vee K} \lim f = l$. Again as $I \subseteq I \vee K $ so using the Proposition \ref{6chpro1}(i) we have $I_\nu^{I \vee K} \lim f = l$ implies $(I\vee K)_\nu^{I \vee K} \lim f = l$.  Hence the first part of the proof is done. \\
For 2nd part of this proof, let us assume the contrary. Let  $K_\nu^{I}$-limit exist and $K_\nu^{I} \lim f = m$, where $l\ne m$. Now as $I \subseteq I \vee K $ so using the Proposition \ref{6chpro1}(ii) we have $K_\nu^{I \vee K} \lim f = m$. Again as $K \subseteq I \vee K $ so using the Proposition \ref{6chpro1}(i) we have $K_\nu^{I \vee K} \lim f = m$ implies $(I\vee K)_\nu^{I \vee K} \lim f = m$ which contradicts the fact of uniqueness of $I_\nu^{K}$-limit. So our assumption is wrong. Therefore $K_\nu^{I} \lim f = l$.

(ii)  As $I \subseteq I \vee K $ so using the Proposition \ref{6chpro1}(ii) we have $I_\nu^{K} \lim f = l$ implies $(I\vee K)^K_\nu \lim f = l$. Again as $K \subseteq I \vee K $ so using the Theorem \ref{6chth2}(i) we have $(I\vee K)^K_\nu \lim f = l$ implies $(I\vee K)_\nu \lim f = l$. Hence the first part of the proof is done. \\
For 2nd part of this proof, let us assume the contrary. Let  $K_\nu^{I}$-limit exist and $K_\nu^{I} \lim f = m$, where $l\ne m$. Now As $K \subseteq I \vee K $ so using the Proposition \ref{6chpro1}(i) we have  $K_\nu^{I} \lim f = m$ implies $(I\vee K)^I_\nu \lim f = m$. Again as $I \subseteq I \vee K $ so using the Theorem \ref{6chth2}(i) we have $(I\vee K)^I_\nu \lim f = m$ implies $(I\vee K)_\nu \lim f = m$, which contradicts the fact of uniqueness of $I_\nu^{K}$-limit. So our assumption is wrong. Therefore $K_\nu^{I} \lim f = l$.
\end{proof}
\section{Basic properties of $I_\nu^K$-convergence in PN spaces}
\begin{theorem}
Let $I\vee K$ be a nontrivial ideal on a non empty set $S$ and let $P$ be a PN Space. Then $I_\nu^K$-convergence of a function $f:S\rightarrow P$ has a unique $I_\nu^K$-limit.
\end{theorem}
\begin{proof}
If possible consider that the function $f$ has two distinct $I_\nu^K$-limits say $x$ and $y$. Given $\epsilon(>0)$ and $t(>0)$. Choose $r>0$ such that $(1-r)*(1-r)\geq 1-\epsilon$. Since $f$ has $I_\nu^K$-limit $x$, so there exists a set $M_1\in F(I)$ such that the function $g:S \rightarrow P$ given by  
\[g(s)=\left\{\begin {array}{ll}
f(s) & \mbox{if $s\in M_1$} \\
x & \mbox{if $s\notin M_1$}
\end{array}
\right. \] 
is $K_\nu$-convergent to $x$.
Now for every $t>0$ and $r>0$ the set
$M_{\nu,x}^g=\{s\in S:\nu_{g(s)-x}(t)>1-r\}\in F(K)$. Let $M_{\nu,x}^f=\{s\in S:\nu_{f(s)-x}(t)>1-r\}$. Then, $M_{\nu,x}^g=(S\setminus M_1)\cup M_{\nu,x}^f=S \setminus (M_1\setminus M_{\nu,x}^f)\in F(K)$
i.e. $M_1\setminus M_{\nu,x}^f\in K$ or $M_1\setminus N_1\in K$ where $N_1= M_{\nu,x}^f$.\\
Similarly, if $f$ has $I_\nu^K$-limit $y$, then there exists a set $M_2\in F(I)$ such that $M_2\setminus M_{\nu,y}^f)\in K$ or $M_2\setminus N_2\in K$ where $N_2=M_{\nu,y}^f$
So,
\begin{equation}\label{6ch4}
(M_1\setminus N_1)\cup (M_2\setminus N_2)\in K
\end{equation}
Now let $p\in (M_1\cap M_2)\cap(N_1\cap N_2)^c = (M_1\cap M_2)\cap (N_1^c\cup N_2^c)=((M_1\cap M_2)\cap N_1^c)\cup ((M_1\cap M_2)\cap N_2^c)$ i.e either $p\in (M_1\cap M_2)\cap N_1^c\subset M_1\cap N_1^c$ or $p\in ((M_1\cap M_2)\cap N_2^c)\subset M_2\cap N_2^c $.
Therefore, $(M_1\cap M_2)\cap(N_1\cap N_2)^c\subset (M_1\cap N^c_1)\cup(M_2\cap N^c_2)\in K$ [from the equation (\ref{6ch4})]. Thus $(M_1\cap M_2)\cap(N_1\cap N_2)^c\in K$ i.e. $(M_1\cap M_2)\setminus (M_{\nu,x}^f\cap M_{\nu,y}^f)\in K$. 
If $ (M_{\nu,x}^f\cap M_{\nu,y}^f)=\phi$ then $M_1\cap M_2\in K$ i.e.
\begin{equation}\label{6ch6}
S\setminus (M_1\cap M_2)\in F(K)
\end{equation} 
As $M_1, M_2\in F(I)$ then
\begin{equation}\label{6ch7}
M_1\cap M_2\in F(I)
\end{equation} 
Since $I\vee K$ is non-trivial then the dual filter $F(I\vee K)=\{G\cap H:G\in F(I),H\in F(K)\}$. Now from the equation \ref{6ch6} and equation \ref{6ch7} we get $\phi\in F(I\vee K)$ which is a contradiction. So $ (M_{\nu,x}^f\cap M_{\nu,y}^f)\neq \phi$. So we have an element $k\in  (M_{\nu,x}^f\cap M_{\nu,y}^f)$. Now 
$$\nu_{x-y}(t)\geq \nu_{f(k)-x}(\frac{t}{2})*\nu_{f(k)-y}(\frac{t}{2})>(1-r)*(1-r)$$. 
Since $(1-r)*(1-r)\geq 1-\epsilon$ it follows that $\nu_{x-y}(t)>1-\epsilon$. Since $\epsilon$ is arbitrary so we get $\nu_{x-y}(t)=1$ for all $t>0$ which yields $x=y$. This completes the prove.
\end{proof}
\begin{theorem} \label{6chth10}
Let $I , K$ be two admissible ideals on a non empty set $S$ and $(P,\nu,*)$ be a PN Space. Then Fin$_\nu(S)$-convergence of a function $f:S\rightarrow P$ implies  $I_\nu^K$-convergence.
\end{theorem}
\begin{proof}
Let a function $f:S\rightarrow P$ be Fin$_\nu(S)$-convergent to $x$ then for each $\epsilon(>0)$ and $t(>0)$ the set
$A(\epsilon, t)=\{s\in S:\nu_{f(s)-\xi}(t)\leq 1-\epsilon\}$ is a finite set. As $K$ is an admissible ideal so $A(\epsilon, t)\in K$. Therefore $f$ is $K_\nu $-convergent $x$. Hence using Lemma \ref{6ch0}
we get $f$ is $I_\nu ^K$-convergent to $x$.
\end{proof}
\begin{theorem}
Let $I , K, I \vee K$ be non trivial  admissible ideals on a non empty set $S$ and $(P, \nu, *)$ be a PN Space. Consider a function $f:S\rightarrow P$ such that $I_\nu ^K- \lim f(s)= l_1$ and $I_\nu ^K- \lim g(s)= l_2$ for some $s\in S$. Then 
(i) $I_\nu ^K- \lim (f(s)+g(s))= l_1 + l_2$ and 
(ii) $I_\nu ^K- \lim (f(s)g(s)= l_1  l_2$
\end{theorem}
\begin{proof}
Suppose $I^K_\nu$-$\lim f= l_1$. Then by definition of $I^K_\nu$-convergence there exist a set $M_1\in F(I)$ such that the function $u:S\rightarrow P$ given by
\[u(s)=\left\{\begin {array}{ll}
f(s) & \mbox{if $s\in M_1$} \\
l_1 & \mbox{if $s\notin M_1$}
\end{array}
\right. \] \\
is $K_\nu$-convergent to $l_1$.	
Again $I^K_\nu$-$\lim g= l_2$ so by definition of $I^K_\nu$-convergence there exist a set $M_2\in F(I)$ such that the function $v:S\rightarrow P$ given by
\[v(s)=\left\{\begin {array}{ll}
f(s) & \mbox{if $s\in M_2$} \\
l_2 & \mbox{if $s\notin M_2$}
\end{array}
\right. \] \\
is $K_\nu$-convergent to $l_2$.	Now by the Theorem 2.3 of \cite{MuM} the function $u(s)+v(s)$ defined by
\begin{equation}\label{6cheq1}
u(s)+v(s)=\left\{\begin {array}{llll}
f(s) +g(s) & \mbox{if $s\in M_1 \cap M_2$} \\
f(s)+l_2 & \mbox{if $s\in M_1 \setminus M_2$}\\
g(s)+l_2 & \mbox{if $s\in M_2 \setminus M_1$}\\
l_1+l_2 & \mbox{if $s\in (M_1 \cup M_2)^c $}\\
\end{array}
\right. 
\end{equation}
is $K_\nu$-convergent to $l_1+l_2$. In another word , for each $\epsilon(>0)$ and $t(>0)$ the set 
\begin{equation}\label{6cheq2}
\{s\in S:\nu_{u(s)+v(s)-(l_1+l_2)}(t)\leq 1-\epsilon\}\in K
\end{equation}
From the equation \ref{6cheq1} and \ref{6cheq2} we have
\begin{equation}\label{6cheq3}
\{s\in S:\nu_{u(s)+v(s)-(l_1+l_2)}(t)\leq 1-\epsilon\}=
\begin{array}{lll}
\{s\in M_1 \cap M_2:\nu_{f(s)+g(s)-(l_1+l_2)}(t)\leq 1-\epsilon\} &\\
\cup \{s\in M_1 \setminus M_2:\nu_{f(s)-l_1}(t)\leq 1-\epsilon\} & \\
\cup \{s\in M_2 \setminus M_1:\nu_{g(s)-l_2}(t)\leq 1-\epsilon\} &
\end{array}
\end{equation}
It is obvious that $M_1 \cap M_2 \in F(I)$. Now we consider the another function $h:S\rightarrow P$ defined as
\[h(s)=\left\{\begin {array}{ll}
f(s)+g(s) & \mbox{if $s\in M_1 \cap M_2$} \\
l_1+l_2 & \mbox{if $s\notin M_1 \cap M_2$}
\end{array}
\right. \] 
Then from the equation \ref{6cheq2} and \ref{6cheq3} and by definition of $h$ function we have,\\
$\{s\in S:\nu_{h(s)-(l_1+l_2)}(t)\leq 1-\epsilon\}= \{s\in M_1 \cap M_2:\nu_{h(s)-(l_1+l_2)}(t)\leq 1-\epsilon\} \cup \{s\in(M_1 \cap M_2)^c :\nu_{h(s)-(l_1+l_2)}(t)\leq 1-\epsilon\}= \{s\in M_1 \cap M_2:\nu_{h(s)-(l_1+l_2)}(t)\leq 1-\epsilon\} \cup \{s\in(M_1 \cap M_2)^c :\nu_{0}(t)\leq 1-\epsilon\}= \{s\in M_1 \cap M_2:\nu_{f(s)+g(s)-(l_1+l_2)}(t)\leq 1-\epsilon\} \subset \{s\in S:\nu_{u(s)+v(s)-(l_1+l_2)}(t)\leq 1-\epsilon\} \in K $
So $h$ is $K_\nu$-convergent to $l_1+l_2$. Hence $f(s)+g(s)$ is $I_\nu^K$-convergent to  $l_1+l_2$.\\
(ii) We omitted the proof of (ii) as it can be obtained by applying similar technique of the above proof.  
\end{proof}
\begin{theorem}\label{6chth11}
Let $I , K$ be two admissible ideals on a non empty set $S$ and $(P,\nu,*)$ be a PN Space. Then for any function $f:S\rightarrow P$, $I_\nu^*$-convergence implies  $I_\nu^K$-convergence.
\end{theorem}
\begin{proof}
Let a function $f:S\rightarrow P$ be $I_\nu^*$-convergent to $l$ then by definition of $I_\nu ^K $ there exists a set $M\in F(I)$ such that the function $g:S\rightarrow P$ defined by
\[g(s)=\left\{\begin {array}{ll}
f(s) & \mbox{if $s\in M$} \\
l & \mbox{if $s\notin M$}
\end{array}
\right. \] \\
is Fin$(S)$-convergent to $l$ with respect to norm $\nu$. Therefore by Theorem \ref{6chth10} the function $f$ is $K_\nu $-convergent $l$. Hence using Lemma \ref{6ch0}
we have $f$ is $I_\nu ^K$-convergent to $l$.
\end{proof}
\begin{theorem}
Let $I , K$ be two admissible ideals on a non empty set $S$ and $(P,\nu,*)$ be a PN Space. If the ideal $I$ has the additive property with respect to $K$ then for any function $f:S\rightarrow P$, $I_\nu$-convergence implies  $I_\nu^K$-convergence.
\end{theorem}
\begin{proof}
If the AP$(I,K)$ condition holds, then from Theorem 3.2 in \cite{MuM} we have $I_\nu$-convergence implies  $I_\nu^*$-convergence. Again from the Theorem \ref{6chth11} we have $I_\nu^*$-convergence implies  $I_\nu^K$-convergence. Hence $f$ is $I_\nu^K$-convergence function. 
\end{proof}

\section{$I_\nu^K$-Cauchy functions}
We can now define in a full generality the idea of Cauchy functions and drive some basic observations.
\begin{definition}[cf. \cite{KV}]
Let $(P,\nu,*)$ be a PNS. A function $f:S\rightarrow P$ is called $I$-Cauchy with respect to probabilistic norm $\nu$ or $I_\nu$-Cauchy function if for each $\epsilon (>0)$ and $t(>0)$ there exists an element $m\in S$ such that the set
$$\{s\in S:\nu_{f(s)-f(m)}(t)\leq 1-\epsilon \}\in I$$ 
\end{definition}
\begin{definition}[cf. \cite{KV}]
Let $(P,\nu,*)$ be a PNS. A function $f:S\rightarrow P$ is called $I^*$-Cauchy with respect to probabilistic norm $\nu$ or $I_\nu^*$-Cauchy function  if there is a subset $M\subseteq S$ such that $M\in F(I)$ and the function $f|_M$ is a Cauchy function with respect to probabilistic norm $\nu$.
\end{definition}

\begin{lemma}
Let $(P,\nu,*)$be a PNS and $I$ be an ideal on a set $S$. For a function $f:S\rightarrow P$ following are equivalent.
\item (i) $f$ is $I_\nu$-Cauchy.
\item (ii) For any $\epsilon (>0)$ and $t(>0)$ there exists an element $m\in S$ such that the set $\{s\in S:\nu_{f(s)-f(m)}(t)>1-\epsilon \}\in F(I)$.
\item (iii) For any $\epsilon (>0)$ and $t(>0)$ there exists an element $m\in S$ such that the set $\{s\in S:|\nu_{f(s)-f(m)}(t)-1|\geq \epsilon \}\in I$.
\end{lemma}
\begin{proof}
Suppose that $f$ is $I_\nu$-Cauchy function.  So the proof of  (i) implies (ii) is straightforward. Again consider the set $\{s\in S:|\nu_{f(s)-f(m)}(t)-1|\geq \epsilon \}= \{s\in S:\nu_{f(s)-f(m)}(t)\geq 1+\epsilon \} \cup \{s\in S:\nu_{f(s)-f(m)}(t)\leq 1-\epsilon \} $. Since for any $\epsilon > 0$ the set  $\{s\in S:\nu_{f(s)-f(m)}(t)\geq 1+\epsilon \}=\phi \in I $ and $\{s\in S:\nu_{f(s)-f(m)}(t)\leq 1-\epsilon \} \in I$ (as from (ii) we have  $\{s\in S:\nu_{f(s)-f(m)}(t)> 1-\epsilon \}\in F(I)$ ). Therefore the set $\{s\in S:|\nu_{f(s)-f(m)}(t)-1|\geq \epsilon \}\in I$. Hence (ii) implies (iii). Converse part follows in similar technique. Hence the result holds.  
\end{proof}.
\begin{proposition} \cite{KV}\label{6chpro5.1}
Let $(P,\nu,*)$be a PNS and $I$ be an admissible ideal on a set $S$ and  $f:S\rightarrow P$ be a function. If $f$ is Cauchy function with respect to $\nu$ then $f$ is $I_\nu$-Cauchy function.
\end{proposition}
\begin{theorem} \cite{KV}
Let $(P,\nu,*)$be a PNS and $I$ be an admissible ideal on a set $S$ and  $f:S\rightarrow P$ be a function. If $f$ is $I_\nu^*$-Cauchy function then $f$ is $I_\nu$-Cauchy function.
\end{theorem}
\begin{theorem} \cite{KV}\label{6chth5.1}
In PNS every $I_\nu$-convergent function is $I_\nu$-Cauchy.
\end{theorem}
\begin{theorem} \cite{KV}\label{6chth5.2}
In PNS every $I_\nu^*$-convergent function is $I_\nu^*$-Cauchy.
\end{theorem}
\begin{theorem}\cite{KV}\label{6chth5.4}
Let $(P,\nu,*)$be a PNS and the ideal $I$ on the set $S$ satisfy the AP-condition. If a function $f$ is any $I_\nu$-Cauchy function then $f$ is also $I_\nu^*$-Cauchy function.
\end{theorem}
\begin{lemma}\label{6chl5.01}
Let $I_1$, $I_2$ be ideals on a set $S$ such that $I_1\subseteq I_2$. Let $P$ be a PNS. If $f:S\rightarrow P$ is $I_1$-Cauchy then it is also $I_2$-Cauchy with respect to probabilistic norm $\nu$.
\end{lemma}
\begin{proof}
The proof is straightforward so omitted.
\end{proof}
\begin{definition}
Let $S$ be an arbitrary set and  $(P,\nu,*)$ be a PNS. Let $I, K$ be ideals on the set $S$. A function $f:S\rightarrow P$ is said to be $I_\nu^K$-Cauchy if there is a subset $M\subseteq S$ such that $M\in F(I)$ and the function $f|_M$ is $K_\nu|_M$-Cauchy.
\end{definition}
\begin{note}
If $I$ is an ideal on $S$ and $M\subseteq S$ then we denote by $I|_M$ the trace of the ideal $I$ on the subset $M$ i.e. $I|_M=\{A\cap M:A\in I\}$ and the dual filter is   $F(I|_M)=\{G\cap M: G\in F(I)\}$
\end{note}
If $K=$Fin we obtain the notion of $I_\nu^*$-Cauchy functions. It is relatively easy to see directly from definition that every $I_\nu^K$-convergent function is $I_\nu^K$-Cauchy.
\begin{lemma}\label{6chl5.1}
If $I$ and $K$ are ideals on an arbitrary set $S$ and $P$ be a PM space and a function $f:S\rightarrow P$ is $K_\nu$-Cauchy then it is also $I_\nu^K$-Cauchy.
\end{lemma}
\begin{proof}
	If we take $M=S$ then $M\in F(I)$. In this case $K|_M=K$, hence $f$ is $K_\nu|_M$-Cauchy. This shows that $f$ is $I_\nu^K$-Cauchy.
\end{proof}
\begin{theorem}\label{6chth5.3}
Let $(P,\nu,*)$be a PNS and $I, K$ be two admissible ideals on a non empty set $S$ and  $f:S\rightarrow P$ be a function. If $f$ is $I_\nu^*$-Cauchy function then $f$ is $I_\nu^K$-Cauchy function.
\end{theorem}
\begin{proof}
Let $f$ be a $I_\nu^*$-Cauchy function then from the definition we have there is a subset $B\subseteq S$ such that $B\in F(I)$ and the function $f|_B$ is a Cauchy function with respect to probabilistic norm $\nu$. Again by Proposition \ref{6chpro5.1} we have $f|_B$ is a $K_\nu $-Cauchy function. Therefore $f$ is a $I_\nu^K$-Cauchy function. 
\end{proof}
\begin{theorem}
Let $(P,\nu,*)$be a PNS and $I, K$ be two admissible ideals on a non empty set $S$ and  $f:S\rightarrow P$ be a function. If $f$ is $I_\nu^K$-convergent function then $f$ is $I_\nu^K$-Cauchy function.
\end{theorem}
\begin{proof}
Let $f$ be $I_\nu^K$-convergent function then by definition we have there exists $B \in F(I)$ such that the function $f|_B$ is $K|_B$-convergent function with respect to probabilistic norm $\nu$ where $K|_B=\{A\cap B:A\in K\}$. Again by Theorem \ref{6chth5.1} we have $f|_B$ is $K|_B$-Cauchy with respect to probabilistic norm $\nu$. Therefore $f$ is $I_\nu^K$-Cauchy function.
\end{proof}
\begin{theorem}
Let $(P,\nu,*)$be a PNS and $I, K$ be two admissible ideals on a non empty set $S$ and  $f:S\rightarrow P$ be a function. If $f$ is $I_\nu^*$-convergent function then $f$ is $I_\nu^K$-Cauchy function.
\end{theorem}
\begin{proof}
The proof directly follows from the Theorem \ref{6chth5.2} and Theorem \ref{6chth5.3}.
\end{proof}

\begin{lemma}
Let $I,J,K$ and $L$ be ideals on a set $S$ such that $I\subseteq J$ and $K\subseteq L$ and let $P$ be a PNS. Then for any function $f:S\rightarrow P$, we have\\
(i) $I_\nu^K$-Cauchy ~~~ $\Rightarrow$ ~~~ $J_\nu^K$-Cauchy and\\
(ii) $I_\nu^K$-Cauchy ~~~ $\Rightarrow$ ~~~ $I_\nu^L$-Cauchy.
\end{lemma}
\begin{proof}
(i) If $f:S\rightarrow P$ is $I_\nu^K$-Cauchy then there is a subset $M\in F(I)$ such that $f|_M$ is $K_\nu|_M$-Cauchy. Since $F(I)\subseteq F(J)$, we have $M\in F(J)$. This means that $f$ is also $J_\nu^K$-Cauchy.\\
(ii) As $K\subseteq L$ implies $K|_M\subseteq L|_M$. From Lemma \ref{6chl5.01} we get that if $f|_M$ is $K_\nu|_M$-Cauchy then it is also $L_\nu|_M$-Cauchy i.e. $f$ is $I_\nu^L$-Cauchy.
\end{proof}
\begin{theorem}
Let $f:S\rightarrow P$ be a function, $I, K$ be ideal on the arbitrary set $S$ and $P$ be a PM space. then\\
(i) $f$ is $I_\nu$-Cauchy if and only if it is $I_\nu^I$-Cauchy. and\\
(ii) $f$ is $I_\nu^K$-Cauchy if and only if it is $(I\vee K)_\nu^K$-Cauchy.
\end{theorem}
\begin{proof}
(i) Suppose that $f$ is $I_\nu$-Cauchy. Then for $M=S$ we have $I|_M=I$ which implies that $f|_M=f$ is also $I_\nu^I$-Cauchy.\\
Conversely, let $f$ is $I_\nu^I$-Cauchy so there is a set $M\in F(I)$ such that $f|_M$ is $I_\nu|_M$-Cauchy. Then for each $\epsilon (>0)$ and $t(>0)$ there exists an element $m\in S$ such that the set $\{s\in S:\nu_{f(s)-f(m)}(t)> 1-\epsilon \}\cap M\in F(I|_M)$. Let $C=\{s\in S:\nu_{f(s)-f(m)}(t)> 1-\epsilon \}$. So there exists $G\in F(I)$ such that $C=G\cap M$. Clearly $G\cap M\in F(I)$ and $G\cap M\subseteq \{s\in S:\nu_{f(s)-f(m)}(t)> 1-\epsilon \}$. So $\{s\in S:\nu_{f(s)-f(m)}(t)> 1-\epsilon \}\in F(I)$\\
(ii) Suppose that $f$ is $I_\nu^K$-Cauchy. Then there is a set $M\in F(I)$ such that $f|_M$ is $K_\nu|_M$-Cauchy. Clearly if $M\in F(I)$ then $M\in F(I\vee K)$. Therefore $f$ is also $(I\vee K)_\nu^K$-Cauchy.\\
Conversely, let $f$ is $(I\vee K)_\nu^K$-Cauchy i.e. there is a set $M\in F(I\vee K)$ such that $f|_M$ is $K_\nu|_M$-Cauchy. Then for each $\epsilon (>0)$ and $t(>0)$ there exists an element $m\in S$ such that the set $\{s\in S:\nu_{f(s)-f(m)}(t)> 1-\epsilon \}\cap M\in F(K|_M)$. Let $C=\{s\in S:\nu_{f(s)-f(m)}(t)> 1-\epsilon \}$ So there exists $G\in F(K)$ such that
$$C\cap M=G\cap M$$
Since $M\in F(I\vee K)$, then $M=M_1\cap M_2$ for some $M_1\in F(I)$ and $M_2\in F(K)$. Now we have $$C\cap M_1\supseteq C\cap M=(G\cap M_2)\cap M_1$$
Since $G\cap M_2\in F(K)$, this shows that $C\cap M_1\in F(K|_{M_1})$. So $f$ is $I_\nu^K$-Cauchy.
\end{proof}
\begin{theorem}
Let $I , K$ be two admissible ideals on a non empty set $S$ and $(P,\nu,*)$ be a PN-space and the ideal $I$ has the additive property with respect to $K$ i.e. AP$(I,K)$ condition holds. If a function $f:S\rightarrow P$ is any $I_\nu$-Cauchy function then $f$ is also $I_\nu^K$-Cauchy function.
\end{theorem}
\begin{proof}
If the AP$(I,K)$ condition holds, then from Theorem \ref{6chth5.4} we have a function $f:S\rightarrow P$ is any $I_\nu$-Cauchy function then $f$ is also $I_\nu^*$-Cauchy function. Again from the Theorem \ref{6chth5.3} we have  $f$ is $I_\nu^*$-Cauchy function then $f$ is $I_\nu^K$-Cauchy function. Hence $f$ is $I_\nu^K$-Cauchy function. 
\end{proof}

\section{$I_\nu^K$-Limit points}
Since we are dealing with functions, we modify the definition of $I_{\nu}$-limit points which is given in \cite{MuM} in the following way:
\begin{definition}
Let $f:S\rightarrow P$ be a function and $I$  be non-trivial ideal of $S$. Then $q\in P$ is called a $I_\nu$-limit point of $f$ if there exists a set $M\subset S$ such that $M\notin I$ and the function $g:S\rightarrow P$ defined by
\[g(s)=\left\{\begin {array}{ll}
f(s) & \mbox{if $s\in M$} \\
q & \mbox{if $s\notin M$}
\end{array}
\right. \] \\
is Fin$(S)$-convergent to $q$ with respect to probabilistic norm $\nu$.
\end{definition}
In the definition of $I_\nu^K$-limit point we simply replace the ideal Fin by an arbitrary ideal $K$ on the set $S$.
\begin{definition}
Let $f:S\rightarrow P$ be a function and $I, K$  be two non-trivial ideals of $S$. Then $q\in P$ is called a $I_\nu^K$-limit point of $f$ if there exists a set $M\subset S$ such that $M\notin I, K$ and the function $g:S\rightarrow P$ defined by
\[g(s)=\left\{\begin {array}{ll}
f(s) & \mbox{if $s\in M$} \\
q & \mbox{if $s\notin M$}
\end{array}
\right. \] \\
is $K_\nu$-convergent to $q$.
\end{definition}
We denote respectively by $I_\nu(\Lambda_s(f))$ and $I^K_\nu(\Lambda_s(f))$ the collection of all $I_\nu$ and $I_\nu^K$-limit points of $f$.\\
If the function $f:S\rightarrow P$ is $I^K_\nu$-convergent to $p$ then , by definition, there exists $M\in F(I)$ such that the function $g:S\rightarrow P$ (given in the definition) is  $K_\nu$-convergent o $p$. This set $M\notin I$, since $I$ is nontrivial. Also this set $M\notin K$. For if $M\in K$ then $S\setminus M \in F(K)$. Since the dual filter $F(I\vee K)$ contains both $F(I)$ and $F(K)$. So $S\setminus M, M\in F(I\vee K)$ which is a contradiction. So, by definition, $p$ is a $I^K_\nu $-limit points of $f$.\\
On the other part, if $p$ is $I^K_\nu $-limit points of $f$ then , by definition,  there exists a set $M$ not belonging to $I$ and $K$ such that the function $g:S\rightarrow P$ (given in definition) is  $K_\nu $-convergent to $p$. In this case $M$ may not belong to $F(I)$. So $f$ may not be $I^K_\nu $-convergent to $p$.

\begin{theorem}
If $K$ is an admissible ideal and $K\subset I$ then $I_\nu(\Lambda_s(f))\subset I^K_\nu (\Lambda_s(f))$
\end{theorem}
\begin{proof}
Let $q\in I_\nu(\Lambda_s(f))$. Since $q$ is a $I_\nu$-limit point of the function $f:S\rightarrow P$, then there exists a set $M\notin I$ such that and the function $g:S\rightarrow P$ defined by
\[g(s)=\left\{\begin {array}{ll}
f(s) & \mbox{if $s\in M$} \\
q & \mbox{if $s\notin M$}
\end{array}
\right. \] \\
is Fin$(S)$-convergent to $q$ with respect to probabilistic norm $\nu$. So for each $\epsilon (>0)$ and $t(>0)$ the set $A(\epsilon, t)=\{s\in S:\nu_{g(s)-q}(t)\leq 1-\epsilon \}$ is a finite set. So $A(\epsilon, t)\in K$, [as $K$ is an admissible ideal.] Therefore $g$ is $K_\nu$-convergent function. Again $M\notin I$ and $K\subset I$ so $M\notin I,K$. Thus $q$ is $I_\nu^K$-limit point of $f$ i.e. $q\in I_\nu^K(\Lambda_s(f))$. Hence the theorem is proved.
\end{proof}
\begin{note}
If $I$ is an admissible ideal and $I\subset K$ then $K_\nu(\Lambda_s(f))\subset I_\nu^K(\Lambda_s(f))$
\end{note}

\section{Acknowledgment}
The authors are grateful to DST, Govt of India for providing fund through FIST program in the Dept. of Mathematics, B.U. during the preparation of this work.


\begin{thebibliography}{9}
\bibitem{BM} M. Balcerzak, K. Dems, A. Komisarski, \textit{Statistical convergence and ideal convergence for sequences of functions}, J. Math. Anal. Appl., \textbf{328} (2007), 715-729.



\bibitem{BA8} A. K. Banerjee, A. Paul, \textit{ $I$-divergence and $I^*$-divergence in cone matric spaces}, Asian-European Journal of Mathematics, \textbf{13(08)} (2020), 2050139.

\bibitem{BA7} A. K. Banerjee, A. Paul, \textit{ On $I$ and $I^*$-Cauchy condition in $C^*$-algebra valued matric spaces}, Korean J Math, \textbf{29(3-4)} (2021), 621-629.

\bibitem{BA0.1} A. K. Banerjee, M. Paul, \textit{Strong-$I^K$-Convergence in Probabilistic metric Spaces}, Iranian Journal of Mathematics Sciences and Informatics, \textbf{17(2)} (2022), 273-288.

\bibitem{BA0} A. K. Banerjee,  M. Paul, \textit{A Note on $I^K$ and $I^{K^*}$-Convergence in Topological Spaces}, arXiv: 1807.11772v1(2018)


\bibitem{BA0.2} A. K. Banerjee, M. Paul,  \textit{Weak and weak* $I^K$-Convergence in normed spaces}, arxiv:1811.06707v1




\bibitem{DP1} P. Das, S. K. Ghosal, \textit{Some further results on I-Cauchy sequences and condition (AP)}, Computers and Mathematics with Applications, \textbf{59} (2010), 2597-2600.



\bibitem{DP2} P. Das, S. K. Ghosal, \textit{On I-Cauchy nets and completeness}, Topology and its Applications, \textbf{157} (2010), 1152-1156. 

\bibitem{DP3} P. Das, P. Kostyrko, P. Malik, W. Wilczy\' nski, \textit{$I$ and $I^*$-convergence of double sequences}, Math. Slov., \textbf{58(5)} (2008), 605-620.

\bibitem{DP4} P. Das, M. Sleziak, V. Toma, \textit{$I^K$-Cauchy functions}, Topology and its Applications, \textbf{173} (2014), 9-27.

\bibitem{DP5} P. Das, S. K. Ghosal, \textit{Some further results on I-Cauchy sequences and Condition(AP)}, Comput. Math. Appl., \textbf{59(8)} (2010), 2597-2600.


\bibitem{DmK} K. Dems, On \textit{I-Cauchy sequences}, Real Analysis Exchange, \textbf{30(1)} (2004/2005), 123-128.


\bibitem{FH} H. Fast, \textit{Sur la convergence statistique}, Colloq. Math., \textbf{2} (1951), 241-244.


\bibitem{FJA1} J. A. Fridy, \textit{On statistical convergence}, Analysis, \textbf{5} (1985), 301-313.


\bibitem{GLB} B. L. Guillen, C. Sempi, \textit{Probalistic norms and convergence of random variables}, J. Math. Anal. Appl., \textbf{280} (2003), 9-16.

\bibitem{HB} B. Hazarika, \textit{Generalized ideal convergence in probabilistic normed spaces}, Journal of Classical Analysis, \textbf{3}(2) (2013), 177-186.


\bibitem{KP0} P. Kostyrko, M. Ma\u caj, T.\u Sal\' at, \textit{Statistical convergence and $I$ -convergence}, Unpublished, http://thales.doa.fmph.uniba.sk/macaj/ICON.pdf.

\bibitem{KP1} P. Kostyrko, T. \u{S}al\'at, W. Wilczy\'nski, \textit{I-convergence}, Real Analysis Exchange, \textbf{26(2)} (2000/2001), 669-686.


\bibitem{KS} S. Karakus, \textit{Statistical convergence on probalistic normed
	spaces}, Mathematical Communications, \textbf{12} (2007), 11-23.

\bibitem{KV} V. Kumar, K. Kumar, \textit{ On $I$ and $I^*$-Cauchy sequences in probabilistic normed spaces}, Mathematical Sciences, \textbf{2(1)} (2008), 47-58.


\bibitem{LB2} B. K. Lahiri, P. Das, \textit{I and  $I^{*}$-convergence in topological spaces}, Math. Bohemica, \textbf{130(2)} (2005), 153-160.


\bibitem{MM1} M. Macaj, M. Sleziak, \textit{$I^K$-convergence}, Real Analysis Exchange, \textbf{36(1)} (2010/2011), 177-194.


\bibitem{MK} K. Menger, Statistical metrics, \textit{Proceedings of the National Academy of Sciences of the United States of America}, \textbf{28} (1942), 535-537.


\bibitem{MuM} M. Mursaleen, S. A. Mohiuddine, \textit{On ideal convergence in probabilistic normed spaces}, Math. Slov., \textbf{62(1)} (2012), 49-62.
\bibitem{MuM1} M. Mursaleen, O. H. H. Edely, \textit{Statistical convergence of double sequences}, J. Math. Anal. Appl., \textbf{288} (2003), 223-231.
\bibitem{MuM2} M. Mursaleen, \textit{$\lambda$-statistical convergence}, Mathematica Slovaca, \textbf{50(1)} (2000), 111-115.

\bibitem{NA} A. Nabiev, S. Pehlivan, M. Gurdal, \textit{On I-Cauchy sequences}, Taiwanese J. Math., \textbf{11(2)} (2007), 569-576.

\bibitem{PS} S. K. Pal, \textit{On $I$ and $I^K$-Cauchy nets and completeness}, Sarajevo Journal of Mathematics, \textbf{10(23)} (2014), 247-255.


\bibitem{RSRM} M. R. S. Rahmat, \textit{Ideal convergence in probabilistic normed spaces}, Intr. Jour. Stat. Econ., \textbf{3(9)} (2009), 67-75.


\bibitem{SB} B. Schweizer, A. Sklar, Probabilistic Metric Spaces, \textit{New York: Elsevier Science Publishing Co.}, 1983.
\bibitem{SB1} B. Schweizer, A. Sklar, \textit{Statistical metric spaces}, Pacific J. Math., \textbf{10(1)} 1960, 313-334.

\bibitem{SA} A. N. Serstnev, \textit{Random normed spaces: Problems of completeness}, Kazan. Gos. Univ.Ucen. Zap., \textbf{122}, (1962), 3-20.

\bibitem{SH} H. Steinhaus, \textit{Sur la convergence ordinaire et la convergence asymptotique}, Colloq. Math., \textbf{2} (1951), 73-74.




\end{thebibliography}
\end{document}